\documentclass[11pt]{article}
\usepackage{amsmath,amssymb,amsfonts,amsthm}
\usepackage{mathtools} 
\usepackage{bbm}
\usepackage{titling}
\usepackage{enumitem}
\usepackage[style=alphabetic,doi=false,isbn=false,url=false,eprint=false]{biblatex}

\usepackage{tikz}
\definecolor{emerald}{rgb}{0.31, 0.78, 0.47}
\definecolor{bleudefrance}{rgb}{0.19, 0.55, 0.91}
\usepackage{hyperref}
\hypersetup{colorlinks=true,linkcolor=bleudefrance,citecolor=emerald}
\usepackage{cleveref}

\usepackage{fullpage}
\usepackage{setspace}
\newcommand{\p}{\mathbb{P}}
\newcommand{\e}{\mathbb{E}}
\newcommand{\eps}{\varepsilon}

\newcommand{\bra}[1]{\left(#1\right)}
\newcommand{\sqbra}[1]{\left[#1\right]}
\newcommand{\cubra}[1]{\left\{#1\right\}}

\newcommand{\den}[1]{\left\lVert#1\right\rVert}
\newcommand{\seq}{_{n \geq 1}}
\newcommand{\abs}[1]{\left\lvert#1\right\rvert}

\crefname{thm}{Theorem}{Theorems}
\crefname{lem}{Lemma}{Lemmas}
\crefname{clm}{Claim}{Claims}
\crefname{rk}{Remark}{Remarks}
\crefname{prop}{Proposition}{Propositions}
\crefname{defn}{Definition}{Definitions}
\crefname{cor}{Corollary}{Corollaries}
\crefname{conj}{Conjecture}{Conjectures}
\crefname{question}{Question}{Questions}
\crefname{section}{Section}{Sections}

\theoremstyle{plain}
\newtheorem{thm}{Theorem}
\newtheorem*{thm*}{Theorem}
\newtheorem{lem}[thm]{Lemma}
\newtheorem*{lem*}{Lemma}

\newtheorem*{clm*}{Claim}
\newtheorem{cor}[thm]{Corollary}
\newtheorem*{cor*}{Corollary}
\newtheorem{prop}[thm]{Proposition}
\newtheorem*{prop*}{Proposition}

\newtheorem*{conj*}{Conjecture}

\theoremstyle{definition}
\newtheorem{defn}[thm]{Definition}
\newtheorem{defn*}{Definition}

\theoremstyle{remark}

\newtheorem*{rk*}{Remark}
\usepackage{bm}

\newcommand{\pcond}{\bm\hat{\mathbb{P}}}
\newcommand{\econd}{\bm\hat{\mathbb{E}}}

\newcommand{\tfail}{\tau_{\text{fail}}}
\newcommand{\tmoat}{\tau_{\text{moat}}}

\newcommand{\1}{\mathbf{1}}

\crefname{ineq}{inequality}{inequalities}
\creflabelformat{ineq}{#2{\upshape(#1)}#3} 

\crefname{alt}{}{}
\creflabelformat{alt}{#2{\upshape(#1)}#3}

\title{Existence of a percolation threshold on finite transitive graphs}
\author{Philip Easo \\ \normalsize{California Institute of Technology}  \\ \normalsize{\href{peaso@caltech.edu}{peaso@caltech.edu}}}

\date{\vspace{-2em}}

\begin{document}
\maketitle

\begin{abstract}
Let $(G_n)$ be a sequence of finite connected vertex-transitive graphs with volume tending to infinity. We say that a sequence of parameters $(p_n)$ is a \emph{percolation threshold} if for every $\eps > 0$, the proportion $\den{K_1}$ of vertices contained in the largest cluster under bond percolation $\p_p^G$ satisfies both
\[ \begin{split}{}
	\lim_{n \to \infty} \p_{(1+\eps)p_n}^{G_n} \bra{ \den{K_1} \geq \alpha } &= 1 \qquad \text{for some $\alpha > 0$, and}\\
	\lim_{n \to \infty} \p_{(1-\eps)p_n}^{G_n} \bra{ \den{K_1} \geq \alpha } &= 0 \qquad \text{for all $\alpha > 0$}.
\end{split} \]
We prove that $(G_n)$ has a percolation threshold if and only if $(G_n)$ does not contain a particular infinite collection of pathological subsequences of dense graphs. Our argument uses an adaptation of Vanneuville's new proof of the sharpness of the phase transition for infinite graphs via couplings \cite{https://doi.org/10.48550/arxiv.2201.08223} together with our recent work with Hutchcroft on the uniqueness of the giant cluster \cite{https://doi.org/10.48550/arxiv.2112.12778}.
\end{abstract}

\section{Introduction}
Given a graph $G$, build a random spanning subgraph $\omega$ by independently including each edge with a fixed probability $p$. This model $\p_p^G$ is called \emph{(Bernoulli bond) percolation}. In their pioneering work on random graphs, Erd\H{o}s and R\'{e}nyi \cite{MR125031} proved that when $G$ is the complete graph on $n$ vertices, percolation has a phase transition: as we increase $p$ from $\frac{1-\eps}{n}$ to $\frac{1+\eps}{n}$ for any fixed $\eps > 0$, a giant cluster suddenly emerges containing a positive proportion of the total vertices. Since then, there has been much interest in establishing this phenomenon for more general classes of finite graphs. However, as remarked in \cite{MR2599196}, progress has been slow. In this paper, we solve this problem for arbitrary finite graphs that are \emph{(vertex-)transitive}, meaning that for any two vertices $u$ and $v$, there is a graph automorphism mapping $u$ to $v$. This includes all Cayley graphs, and in particular, the complete graphs, hypercubes, and tori.

Our setting of percolation on finite transitive graphs places us at the intersection of two well-established fields. Loosely speaking, one of these began in combinatorics with the work of Erd\H{o}s and R\'{e}nyi \cite{MR120167,MR125031}, whereas the other began in mathematical physics with the work of Broadbent and Hammersley \cite{MR91567}. In the former, a subgraph of a finite graph is said to \emph{percolate} if its largest cluster contains a positive proportion of the total vertices, whereas in the latter, a subgraph of an infinite transitive graph is said to \emph{percolate} if its largest cluster is infinite. This leads to two different definitions of what it means to \emph{have a percolation phase transition}. Let us start by making these precise.

Here are the graph-theoretic conventions we will be using throughout: The \emph{volume} of a graph $G = (V,E)$ is simply the number of vertices $\abs{V}$. We label the clusters (i.e.\! connected components) of a spanning subgraph of $G$ in decreasing order of volume by $K_1,K_2,\dots$. In a slight abuse of notation, we also write $K_v$ for the cluster containing a given vertex $v$. The \emph{density} of a cluster $K$ is $\den{K} := \frac{\abs{K}}{\abs{V}}$, the proportion of vertices contained in $K$.

Now let $(G_n)$ be a sequence of finite graphs with volume tending to infinity. Following Bollob\'{a}s, Borgs, Chayes, and Riordan \cite{MR2599196}, we say that $(G_n)$ has a percolation phase transition if there is a sequence of parameters $(p_n)$ such that for every $\eps > 0$, both\footnote{We will use the convention that $\p_p^G := \p_1^G$ if $p > 1$ and $\p_p^G := \p_0 ^G$ if $p < 0$.}
\[ \begin{split}{}
	\lim_{n \to \infty} \p_{(1+\eps)p_n}^{G_n} \bra{ \den{K_1} \geq \alpha } &= 1 \qquad \text{for some $\alpha > 0$, and}\\
	\lim_{n \to \infty} \p_{(1-\eps)p_n}^{G_n} \bra{ \den{K_1} \geq \alpha } &= 0 \qquad \text{for all $\alpha > 0$}.
\end{split} \]
This is the subject of our paper: we characterise the existence of a percolation phase transition for finite transitive graphs. Let us mention that the question of a percolation phase transition on general finite graphs is attributed by the above authors of \cite{MR2599196} to Bollob\'{a}s, Kohakayawa, and Łuksak \cite{MR1139488}.

On the other hand, when $G$ is an infinite (locally finite) transitive graph, we define the critical parameter
\[
	p_c := \sup \cubra{p : \p_p^G(\text{there exists an infinite cluster}) = 0 }.
\]
By Kolmogorov's zero-one law, the probability of an infinite cluster under $\p_p^G$ is zero when $p < p_c$ and one when $p > p_c$. So in a trivial sense, $G$ always has a percolation phase transition. The real question is whether $p_c < 1$. (The fact that $p_c > 0$ is obvious by a branching argument.) So in this context, we often say that $G$ has a percolation phase transition to mean that $p_c < 1$. Hutchcroft and Tointon \cite{https://doi.org/10.48550/arxiv.2104.05607} dealt with an analogue of this question for finite transitive graphs with bounded vertex degrees, i.e.\! the question of whether for a given sequence $(G_n)$ of such graphs, there exists $\delta > 0$ such that $\p_{1-\delta}^{G_n}(\den{K_1} \geq \delta) \geq \delta$ for all $n$. This is not the subject of our paper. To avoid any possible confusion, when a sequence of finite graphs $(G_n)$ with volume tending to infinity has a percolation phase transition in the above sense of \cite{MR2599196}, we will instead say that it \emph{has a percolation threshold}, referring to the threshold sequence of parameters $(p_n)$ in the definition.

The main result of this paper is \cref{thm:main} below, which characterises the existence of a percolation threshold on a sequence of finite transitive graphs in terms of the presence of an infinite collection of \emph{molecular subsequences}. We discovered molecular sequences with Hutchcroft in \cite{https://doi.org/10.48550/arxiv.2112.12778} as the only obstacles to the supercritical giant cluster being unique. Interestingly, unlike the usual story for percolation on a new family of graphs (as told in the introduction of \cite{MR2599196}, for example), uniqueness of the supercritical giant cluster came \emph{first}, before the existence of a percolation threshold, and the former is key to our proof of the latter. See \cref{subsec:defining:molecular_seq} for more background. Here we will just recall the definition of a molecular sequence before stating \cref{thm:main}.

\begin{defn} \label{def:molecules}
	Given an integer $m \geq 2$, we say that $(G_n)$ is \emph{$m$-molecular} if it is dense, meaning that ${\liminf_{n \to \infty} \frac{ \abs{E(G_n)}}{\abs{V(G_n)}^2} > 0}$, and there is a constant $C < \infty$ such that for every $n$, there is a set of edges $F_n \subseteq E(G_n)$ satisfying the following conditions:
	\begin{enumerate}[nolistsep]
		\item $G_n \backslash F_n$ has $m$ connected components;
		\item $F_n$ is invariant under the action of $\operatorname{Aut} G_n$;
		\item $\abs{F_n} \leq C \abs{V(G_n)}$.
	\end{enumerate}
	For example, the sequence of Cartesian products of complete graphs $(K_n \square K_m)\seq$ is $m$-molecular. We say that $(G_n)$ is \emph{molecular} if it is $m$-molecular for some $m \geq 2$.
\end{defn}

\begin{thm} \label{thm:main}
A sequence of finite connected transitive graphs with volume tending to infinity has a percolation threshold if and only if it contains an $m$-molecular subsequence for at most finitely many integers $m$.
\end{thm}

The condition that a sequence $(G_n)$ contains $m$-molecular subsequences for infinitely many integers $m$ is extremely stringent. For example, we can rule it out if $(G_n)$ is either sparse or dense, i.e.\! the edge density of $G_n$ either tends to zero or remains bounded away from zero. Indeed, it is clear that a sparse sequence cannot contain any molecular subsequences, but also notice that since every $m$-molecular sequence $(G_n)$ satisfies ${\limsup_{n \to \infty} \frac{ \abs{E(G_n)}}{\abs{V(G_n)}^2} \leq \frac{1}{m}}$, a dense sequence can contain an $m$-molecular subsequence for at most finitely many integers $m$. In particular, our result implies the existence of a percolation threshold for the complete graphs, hypercubes, and tori. Our result is new even under the additional hypothesis that the graphs have uniformly bounded vertex degrees\footnote{One could imagine a family of sequences of such graphs that each has a percolation threshold but such that the $(1\!+\!\eps)$-supercritical giant cluster density is not bounded away from zero over the entire family. Then by diagonalising, we could construct a sequence without a percolation threshold.}, which is particularly relevant to percolation on infinite graphs.

Most previous work on percolation on finite graphs treated specific sequences such as the complete graphs, hypercubes, and tori. Indeed, many authors have remarked how little work has been done on more general classes of finite graphs \cite{MR2073175,MR2155704,MR2599196,MR3156647}. Alon, Benjamini, and Stacey \cite{MR2073175}\footnote{Here Alon, Benjamini, and Stacey also make several conjectures about percolation on finite transitive graphs, which may interest the reader.} studied percolation on expanders with bounded vertex degrees. In particular, they proved that if each graph is $d$-regular for some fixed integer $d$ and has girth tending to infinity, then the sequence has a (constant) percolation threshold at $p_n := 1/(d-1)$. Borgs, Chayes, van der Hofstad, Slade, and Spencer \cite{MR2155704,MR2165583} and Nachmias \cite{MR2570320} analysed the emergence of a cluster with volume of order $\abs{V(G_n)}^{2/3}$ for percolation on finite transitive graphs satisfying the \emph{triangle condition} or a random walk return-probability condition, respectively, both of which enforce mean-field behaviour. Frieze, Krivelevich, and Martin \cite{MR2020308} proved that if a sequence of finite regular graphs is \emph{pseudorandom} (an eigenvalue condition forcing the graph to be like the complete graph), then it has a percolation threshold at $p_n := 1/\deg(G_n)$ where $\deg(G_n)$ is the vertex degree of the $n$th graph. Bollob\'{a}s, Borgs, Chayes, and Riordan \cite{MR2599196} studied percolation on arbitrary sequences of finite graphs $(G_n)$ that are \emph{dense}, meaning that ${\liminf_{n \to \infty} \frac{ \abs{E(G_n)}}{\abs{V(G_n)}^2} > 0}$. The following theorem from their paper will be used in our argument.

\begin{thm}[Bollob\'{a}s, Borgs, Chayes, Riordan 2010] \label{thm:dense}
	Let $(G_n)$ be a sequence of finite connected graphs with volume tending to infinity. Suppose that $(G_n)$ is dense. For each $n$, let $\lambda_n$ be the largest eigenvalue of the adjacency matrix of $G_n$. Then $(1/\lambda_n)_{n \geq 1}$ is a percolation threshold for $(G_n)$.
\end{thm}

Our proof of \cref{thm:main} does not build a percolation threshold by defining a natural candidate for the critical parameter of a finite graph in terms of, say, its vertex degrees or the largest eigenvalue of the adjacency matrix. In our setting, where we have no quantitative assumptions, we are forced to use a softer and more indirect approach. In particular, our proof of \cref{thm:main} says little about the rate at which the percolation probabilities tend to zero or to one. That said, as part of our proof we do obtain explicit lower bounds on the supercritical giant cluster density, which are analogous to the well-known mean-field lower bound for percolation on infinite graphs.

\begin{cor} \label{cor:lower_bound}
Let $(G_n)$ be a sequence of finite connected transitive graphs with volume tending to infinity that does not contain an $m$-molecular subsequence for any $m > M$, where $M$ is some positive integer. Let $(p_n)$ be a percolation threshold, which exists by \cref{thm:main}. Then for every $\eps > 0$,
\[
	\p_{(1+\eps)p_n}^{G_n} \bra{ \den{K_1} \geq \frac{\eps}{M(1+\eps)} - o(1) } = 1-o(1) \quad \text{as $n \to \infty$.}
\]
\end{cor}

Moreover, by simply bounding a percolation threshold $(p_n)$ below by the threshold for being dominated by a subcritical branching process and above by the threshold for connectivity, we obtain the following optimal bounds on its location.

\begin{prop} \label{prop:sharp_bounds}
Let $(G_n)$ be a sequence of finite connected transitive graphs with volume tending to infinity. Let $d_n$ denote the vertex degree of $G_n$. If $(G_n)$ has a percolation threshold $(p_n)$, then it satisfies
\[
	(1-o(1))\frac{1}{d_n - 1} \leq p_n \leq (2+o(1))\frac{\log \abs{V(G_n)}}{d_n} \quad \text{as $n \to \infty$}.
\]
\end{prop}

When $(G_n)$ has a percolation threshold $(p_n)$ and converges locally to an infinite graph $G$, one might ask how the location of $(p_n)$ relates to the critical parameter $p_c$ and the uniqueness threshold $p_u$ of $G$. See Remark 1.6 in \cite{https://doi.org/10.48550/arxiv.2112.12778} for a discussion of this question. Let us simply note that $(p_n)$ typically (but not always) converges to $p_c$. For example, this is the case when $(G_n)$ is a sequence of transitive expanders and $G$ is nonamenable \cite{Benjamini:2011vx}.

We conclude this discussion by explaining how our result relates to the general theory of sharp thresholds. Consider a large collection of independent random bits $b := (b_i)_{1 \leq i \leq n} \in \{ 0,1 \}^n$ each sampled according to the Bernoulli($p$) distribution for some $p \in [0,1]$. For many natural monotone events $A \subseteq \{ 0,1 \}^n$, the probability that $b$ belongs to $A$ has a sharp threshold: it increases from $o(1)$ to $1-o(1)$ as $p$ increases across an interval of width $o(p(1-p))$. There is a general philosophy that a sharp threshold occurs if and only if $A$ is sufficiently \emph{symmetric}/\emph{global}. For example, the event that the Erd\H{o}s-R\'{e}nyi random graph is connected has a sharp threshold, but the event that it contains a triangle and the event that a particular edge is present do not. \Cref{thm:main} can be understood as a kind of extension of this principle to the existence of a phase transition in a statistical mechanics model. Indeed, \Cref{thm:main} says that "having a giant cluster" typically has a threshold around a critical parameter $p$ with a threshold width\footnote{The threshold width fails to be $o(1-p)$ even for simple examples such as the sequence of cycles.} $o(p)$ whenever the underlying graph is \emph{transitive}, which is a symmetry/homogeneity condition. However, we would like to emphasise that "having a giant cluster" is not a well-defined event, so it does not fall within the usual scope of sharp-threshold techniques.

\subsection{Molecular sequences and the supercritical phase} \label{subsec:defining:molecular_seq}

Let $(G_n)$ be a sequence of finite connected transitive graphs with volume tending to infinity. If $(G_n)$ has a percolation threshold $(p_n)$, then a sequence of parameters $(q_n)$ is called \emph{supercritical} if there exist $\eps >0$ and $N < \infty$ such that for all $n \geq N$ satisfying $q_n < 1$, we have
\[
	q_n \geq (1+\eps) p_n.
\]
As in \cite{https://doi.org/10.48550/arxiv.2104.05607,https://doi.org/10.48550/arxiv.2112.12778}, we generalise this definition to the situation that there may or may not be a percolation threshold by saying that a sequence of parameters $(q_n)$ is \emph{supercritical} if there exist $\eps > 0$ and $N < \infty$ such that for all $n \geq N$ satisfying $q_n <1$, we have
\[
	\p_{(1-\eps)q_n} \bra{ \den{K_1} \geq \eps } \geq \eps.
\]
(The `$q_n < 1$' condition is a technicality that guarantees that $(G_n)$ always admits at least one supercritical sequence of parameters, namely the sequence $(p_n)$ with $p_n :=1$ for all $n$.)

In our work with Hutchcroft \cite{https://doi.org/10.48550/arxiv.2112.12778}, we showed that not having a molecular subsequence is the geometric counterpart to the supercritical giant cluster being unique. Here is the precise result from that paper.

\begin{defn}
	We say that $(G_n)$ has the \emph{supercritical uniqueness property} if for every supercritical sequence of parameters $(q_n)$,
	\[
		\lim_{n \to \infty} \p_{q_n}^{G_n} \bra{ \den{K_2} \geq \alpha } = 0 \qquad \text{for all $\alpha > 0$}.
	\]
\end{defn}

\begin{thm}[Easo and Hutchcroft, 2021] \label{thm:sup}
A sequence of finite connected transitive graphs with volume
tending to infinity has the supercritical uniqueness property if and only if it does not contain a molecular
subsequence.
\end{thm}

A crucial step in the proof of the above theorem is that if $(G_n)$ does not contain a molecular subsequence, then it has the \emph{sharp-density property}. To prove \cref{thm:main}, we only need to recall that the sharp-density property guarantees that for every density $\alpha \in (0,1]$, constant $\eps > 0$, and supercritical sequence of parameters $(p_n)$,
\[
	\liminf_{n \to \infty} \p_{p_n}^{G_n} \bra{ \den{K_1} \geq \alpha } > 0 \qquad \text{implies} \qquad \lim_{n \to \infty} \p_{(1+\eps)p_n}^{G_n} \bra{ \den{K_1} \geq \alpha } = 1.
\]
This implies that the event that ``there exists a cluster with density at least $\alpha$'' undergoes a sharp threshold for each fixed $\alpha$. Notice that this does not immediately imply the existence of a percolation threshold. One obstruction could be how these thresholds are spaced: one could imagine that the $\alpha$-density threshold always occurs at $\left(\frac{\alpha}{\sqrt{n}}\right)_{n \geq 1}$, say, in which case there would be no percolation threshold. The implication is not clear even if we additionally require that the $G_n$'s have uniformly bounded vertex degrees: see our footnote on page 3, and see Conjecture 1.2 in \cite{MR2917769} for an analogous situation in the context of expanders.

\subsection{Proof strategy} \label{subsec:proof_strategy}

Most of our work goes into showing that if $(G_n)$ does not contain a molecular subsequence, then it has a percolation threshold. In this section, we outline this step. To extend this result to the case that $(G_n)$ contains an $m$-molecular subsequence for at most finitely many integers $m$, we apply the same argument to a molecular sequence's constituent sequence of \emph{atoms}. We will deduce that converse as an immediate consequence of a corollary from \cite{https://doi.org/10.48550/arxiv.2112.12778}.

Let $(G_n)$ be a sequence of finite connected transitive graphs with volume tending to infinity. Suppose we want to prove that $(G_n)$ has a percolation threshold. Although "having a giant cluster" is not an event, we could try taking the event $\{ \den{K_1} \geq \delta_n\}$ as a proxy, where $(\delta_n)$ is a sequence tending to zero very slowly. Then as a candidate for the percolation threshold, we could take $(p_n)$ where each $p_n$ is defined to be the unique parameter satisfying
\[
	\p_{p_n}^{G_n} \bra{ \den{K_1} \geq \delta_n } = \frac{1}{2}.
\]
As a sanity check, notice that if $(G_n)$ does have a percolation threshold, then by a diagonal argument, there is a percolation threshold $(p_n)$ that arises in this way.

To prove that our candidate $(p_n)$ is in fact a percolation threshold, we need two ingredients. The first ingredient is that the emergence of a cluster of any \emph{constant} density has a threshold about a critical parameter $p$ with a threshold width $o(p)$. This immediately handles the subcritical half of our task: since $\lim_{n \to \infty}\delta_n = 0$, it guarantees that $\lim_{n \to \infty} \p_{(1-\eps)p_n}^{G_n} \bra{ \den{K_1} \geq \alpha } = 0$ for every $\alpha,\eps > 0$. As mentioned in \cref{subsec:defining:molecular_seq}, the existence of these constant-density thresholds is implied by the sharp-density property, which holds whenever $(G_n)$ has no molecular subsequences.

The second ingredient is a universal lower bound on the supercritical giant cluster density. This says that for every $\eps > 0$, there exists $\delta > 0$ such that for every sequence of parameters $(p_n)$, if $\liminf_{n \to \infty} \e_{p_n}^{G_n} \den{K_1} > 0$, then $\liminf_{n \to \infty} \e_{(1+\eps)p_n}^{G_n} \den{K_1} \geq \delta$. Together with the first ingredient, this ensures that by taking $(\delta_n)$ to decay slowly enough, for every $\eps > 0$, there exists $\delta > 0$ with $\lim_{n \to \infty} \p_{(1+\eps)p_n}^{G_n}\bra{\den{K_1} \geq \delta } = 1$. We will prove that this too holds whenever $(G_n)$ has no molecular subsequences.

This second ingredient is reminiscent of the mean-field lower bound from the study of percolation on infinite graphs. This says, for example, that every vertex $v$ in an infinite transitive graph $G$ satisfies
\[
	\p_{(1+\eps)p_c(G)}^G \bra{ K_v \text{ is infinite} } \geq \frac{\eps}{1+\eps}.
\]
Together with the exponential decay of $\abs{K_v}$ throughout the subcritical phase, this forms what is known as the \emph{sharpness of the phase transition} for percolation on infinite graphs. This foundational result was first proved in \cite{MR852458,MR874906,MR894542} but has recently been reproved by more modern arguments in \cite{MR3477351,MR3898174,MR4408005,https://doi.org/10.48550/arxiv.2201.08223}. The main obstacle to adapting these proofs to our finite setting is that they concern the event that the cluster at a vertex reaches a certain distance or exceeds a certain volume, whereas we care about the cluster's volume as a proportion of the total vertices.

Moreover, while the sharpness of the phase transition is completely general - applying to all infinite transitive graphs - there is no universal lower bound on the supercritical giant cluster density that applies to every finite transitive graph\footnote{Consider the sequences $(K_n \square C_m)\seq$ for each $m \geq 3$, where $K_n$ is a complete graph and $C_m$ is a cycle.}. So to adapt one of these proofs to our setting, we need to include information about the underlying sequence of graphs in the argument itself, for example, that it has the supercritical uniqueness property, or equivalently, that it has no molecular subsequences.

Very recently, Vanneuville \cite{https://doi.org/10.48550/arxiv.2201.08223} gave a new proof of the sharpness of the phase transition via couplings. Unlike previous arguments, this one does not rely on a differential inequality. Vanneuville's key insight was that by using an exploration process, we can upper bound the effect of conditioning on a certain decreasing event, namely the event that a vertex's cluster does not reach a certain distance, by the effect of slightly decreasing the percolation parameter. Our strategy is to apply this argument but with the event that a vertex's cluster has small density.

Rather than building an exact monotone coupling of the conditioned percolation measure and the percolation measure with a smaller parameter, which is impossible, we will construct a coupling that is monotone outside of an error event. Then under the additional hypothesis that $(G_n)$ has the supercritical uniqueness property, we will prove that this error event has probability tending to zero. Just as Vanneuville's coupling immediately yields the mean-field lower bound, our \emph{approximately monotone} coupling will tell us that for every sequence of parameters $(p_n)$ and every constant $\eps > 0$, if $\liminf_{n \to \infty} \e_{p_n}^{G_n} \den{K_1} > 0$, then
\[
	\liminf_{n \to \infty} \e_{(1+\eps)p_n}^{G_n} \den{K_1} \geq \frac{\eps}{1+\eps}.
\]

Let us mention that for this step - establishing a universal lower bound on the supercritical giant cluster density when we have the supercritical uniqueness property - it is possible to instead adapt Hutchcroft's proof of sharpness from \cite{MR4408005}, rather than Vanneuville's new proof. The adaptation that we found of Hutchcroft's proof is more involved than the argument presented here. For example, it invokes the universal tightness result from \cite{Hutchcroft:2021vn}. Invoking this auxiliary result also has the consequence that the universal lower bound we ultimately obtain is weaker than the mean-field bound established here.

\section{The coupling lemma} \label{sec:coupling_lemma}

In this section, we control the effect on percolation of conditioning on the event that a vertex's cluster has small density. In particular, we prove that the conditioned measure \emph{approximately} stochastically dominates percolation of a slightly smaller parameter. As mentioned in \cref{subsec:proof_strategy}, the argument in this section is inspired by \cite{https://doi.org/10.48550/arxiv.2201.08223}.

\begin{lem} \label{lem:coupling}
Let $G=(V,E)$ be a finite connected transitive graph with a distinguished vertex $o$. Let $p \in (0,1)$ be a parameter and $\alpha \in (0,1)$ a density. Define
\[
	\theta := \e_p \den{K_1}, \quad h := \p_p \bra{ \den{K_1} < \alpha \text{ or } \den{K_2} \geq \frac{\alpha}{2} }, \quad \delta := \frac{2h^{1/2}}{1-\theta-h},
\]
and assume that $\theta+h < 1$ (so that $\delta$ is well-defined and positive). Then there is an event $A$ with $\p_p \bra{ A  \mid \den{K_o} < \alpha } \leq h^{1/2}$ such that
\[
	\p_{ \bra{1 - \theta - \delta}p } \leq_{\mathrm{st}} \p_p \bra{\omega \cup \1_A = \boldsymbol{\cdot}\; \mid \den{K_o} < \alpha},
\]
where $\leq_{\mathrm{st}}$ denotes stochastic domination with respect to the usual partial order $\preceq$ on $\{0,1\}^E$, and $\1_A$ denotes the random configuration with every edge open on $A$ and every edge closed on $A^c$.
\end{lem}

\begin{proof}
To lighten notation, set $\pcond := \p_p \bra{ \; \boldsymbol{\cdot}  \mid \den{K_o} < \alpha }$ and $q := \bra{1 - \theta - \delta}p$. Our goal is to construct an approximately monotone coupling of $\p_q$ and $\pcond$. We will build this in the obvious way: by fixing an exploration process and building samples of $\p_q$ and $\pcond$ in terms of a common collection of $E$-indexed uniform random variables. The rest of the proof consists in controlling the failure of this coupling to be monotone.

Fix an enumeration of the edge set $E$. Recall the following standard method for exploring a configuration $\omega$ from $o$: Start with all edges unrevealed. Iteratively reveal the unrevealed edge of smallest index that is connected to $o$ by an open path of revealed edges until there are none. Then iteratively reveal the unrevealed edge of smallest index from among all remaining unrevealed edges until there are none. Let $\rho_1,\dots,\rho_{\abs{E}}$ be the sequence of edges as they are revealed by this process. Let $\bra{\mathcal F_t}_{0 \leq t \leq \abs{E}}$ be the filtration associated to this exploration, i.e.
\[
	\mathcal F_t := \sigma\bra{ \omega_{\rho_1} ,\dots, \omega_{\rho_t} }.
\]

We now use this exploration to construct a coupling of $\p_q$ and $\pcond$. Let $(U_e)_{e \in E}$ be a collection of independent uniform-[0,1] random variables. Recursively define $\omega_{\rho_t} := \1_{U_{\rho_t} \leq \pcond\bra{\rho_t \text{ open } \mid \mathcal F_{t-1} }}$ for every $t$ to obtain a configuration $\omega$ with law $\pcond$, and simply take $\bra{\1_{U_e \leq q}}_{e \in E}$ for a configuration with law $\p_q$. This coupling is monotone (in the direction we want) on the edges $\rho_1,\rho_2,\dots,\rho_{\tfail}$ where $\tfail$ is the stopping time defined by
\[
	\tfail := \inf \{ t : \pcond \bra{ \rho_{t+1} \text{ open} \mid \mathcal F_t } < q \},
\]
with the convention that $\inf \emptyset := \abs{E}$. So $\bra{\1_{U_e \leq q}}_{e \in E} \preceq \omega$ almost surely when $\tfail = \abs{E}$. Since $\bra{\1_{U_e \leq q}}_{e \in E} \preceq \1_{\tfail < \abs{E}}$ holds trivially when $\tfail < \abs{E}$, we know that
\[
	\p_q \leq_{\mathrm{st}} \pcond \bra{ \omega \cup \1_{\tfail < \abs{E}} = \boldsymbol{\cdot} \; }.
\]
So it suffices to verify that $\pcond \bra{ \tfail < \abs{E} } \leq h^{1/2}$. By definition of $\tfail$, we have the upper bound
\begin{equation} \label[alt]{ineq:trivial_bound_at_tfail}
	\pcond \bra{ \rho_{\tfail+1}\text{ open} \mid \mathcal F_{\tfail} } < q \quad \text{a.s.\! when $\tfail < \abs{E}$}.
\end{equation}
Our first step is to prove a complementary lower bound.

Say that an edge $e$ is \emph{pivotal} if $\den{K_o (\omega \backslash \{ e \} )} < \alpha$ but $\den{K_o (\omega \cup \{ e\})} \geq \alpha$. If $e$ is open and pivotal, then $\den{K_o} \geq \alpha$, which 
is $\pcond$-almost surely impossible. So
\begin{equation} \label{eq:inserting_nonpivotality}
	\pcond \bra{ \rho_{\tfail+1} \text{ open} \mid \mathcal F_{\tfail} } = \pcond \bra{ \rho_{\tfail+1} \text{ open and not pivotal} \mid \mathcal F_{\tfail} } \; \; \text{a.s.\! when } \tfail < \abs{E}.
\end{equation}
Suppose we reveal the edges $\rho_1,\dots,\rho_{\tfail}$ and find that $\tfail < \abs{E}$. Note that $\rho_{\tfail + 1}$ is now almost surely determined. To finish building a sample of $\pcond$, rather than continuing our exploration process, we could first sample every unrevealed edge except $\rho_{\tfail + 1}$, then sample $\rho_{\tfail + 1}$ itself. The first stage will determine whether $\rho_{\tfail + 1}$ is pivotal. If it is not pivotal, then conditioning on the event $\{\den{K_o} < \alpha\}$ will have no effect in the second stage, i.e.\! the conditional probablity that $\rho_{\tfail + 1}$ is open will simply be $p$. So we can rewrite \cref{eq:inserting_nonpivotality} as
\begin{equation} \label{eq:open_to_piv} \begin{split}
	\pcond \bra{ \rho_{\tfail+1} \text{ open} \mid \mathcal F_{\tfail} }
	&= p \pcond \bra{  \rho_{\tfail+1} \text{ not pivotal} \mid \mathcal F_{\tfail} } \\
	&= p \bra{ 1 - \pcond \bra{ \rho_{\tfail+1} \text{ pivotal} \mid \mathcal F_{\tfail} } } \quad \text{a.s.\! when } \tfail < \abs{E}.
\end{split} \end{equation}
As in our argument for \cref{eq:inserting_nonpivotality}, since it is $\pcond$-almost surely impossible for an edge to be both open and pivotal, this further implies that
\begin{equation} \label{eq:final_manipulation}\begin{split}
	\pcond \bra{ \rho_{\tfail+1} \text{ open} \mid \mathcal F_{\tfail} } &= p \bra{ 1 - \pcond \bra{ \rho_{\tfail+1} \text{ closed and pivotal} \mid \mathcal F_{\tfail} } } \quad\text{a.s.\! when } \tfail < \abs{E}.
\end{split}\end{equation}

By combining inequality \cref{ineq:trivial_bound_at_tfail} with \cref{eq:open_to_piv}, we deduce that
\[
	\pcond \bra{ \rho_{\tfail+1} \text{ pivotal} \mid \mathcal F_{\tfail} } > 0 \quad \text{a.s.\! when } \tfail < \abs{E}.
\]
So when $\tfail < \abs{E}$, we can almost surely label the endpoints of $\rho_{\tfail+1}$ by $v_-$ and $v_+$ such that $v_-$ is connected to $o$ by a path of open edges among the revealed edges $\{ \rho_1,\dots, \rho_{\tfail}\}$, whereas $v_+$ is not.

Consider a configuration $\omega$ with $\tfail < \abs{E}$ in which $\rho_{\tfail + 1}$ is closed and pivotal. By the pigeonhole principle, since $\den{ K_o(\omega \cup \{\rho_{\tfail+1}\} ) } \geq \alpha$,
\begin{equation} \label[alt]{alt:pigeonhole}
	\den{K_{v_-}} \geq \frac{\alpha}{2} \quad \text{or} \quad \den{K_{v_+}} \geq \frac{\alpha}{2}.
\end{equation}
Since closing $\rho_{\tfail+1}$ disconnects $K_o(\omega \cup \{ \rho_{\tfail + 1} \})$, the endpoints $v_-$ and $v_+$ belong to distinct clusters of $\omega$. In particular, since $v_- \in K_o$, we know that $v_+ \not\in K_o$. So \cref{alt:pigeonhole} implies that
\[
	\den{K_o} \geq \frac{\alpha}{2} \quad \text{or} \quad \big\lVert K_{v_+}(\omega \backslash \overline{K_o} )\big\rVert \geq \frac{\alpha}{2},
\]
where $\overline{K_o}:= K_o \cup \partial K_o$ and $\partial K_o$ denotes the edge boundary of $K_o$. Now $\omega$ was arbitrary, so by a union bound,
\begin{equation} \label[ineq]{ineq:piv_to_split} \begin{split}
	\pcond \bra{  \rho_{\tfail+1} \text{ closed and pivotal} \mid \mathcal F_{\tfail} } &\leq \pcond \bra{  \den{K_o} \geq \frac{\alpha}{2} \mid \mathcal F_{\tfail} } + \pcond \bra{ \big\lVert K_{v_+}(\omega \backslash \overline{K_o} )\big\rVert \geq \frac{\alpha}{2} \mid \mathcal F_{\tfail}} \\
	& \quad \quad \text{a.s.\! when } \tfail < \abs{E}.
\end{split} \end{equation}

To bound the first term in \cref{ineq:piv_to_split}, notice that $\den{K_o} \geq \frac{\alpha}{2}$ and $\den{K_o} < \alpha$ together imply the bad event $B := \{\den{K_1} < \alpha \text{ or } \den{K_2} \geq \frac{\alpha}{2}\}$. 
So because $\den{K_o} < \alpha$ occurs $\pcond$-almost surely,
\begin{equation} \label[ineq]{ineq:first_term}
	\pcond \bra{  \den{K_o} \geq \frac{\alpha}{2} \mid \mathcal F_{\tfail} } \leq \pcond \bra{ B \mid \mathcal F_{\tfail}} \quad \text{a.s.\! when } \tfail < \abs{E}.
\end{equation}

To bound the second term in \cref{ineq:piv_to_split}, define a new stopping time
\[
	\tmoat := \max \{t : \rho_t \in \overline K_o \}.
\]
(This is defined with respect to the standard exploration described in the second paragraph of the current proof environment, not the modified exploration mentioned below \cref{eq:inserting_nonpivotality}.) Just after we reveal $\rho_{\tmoat}$, since conditioning on the event $\{\den{K_o} < \alpha\}$ no longer has any effect, the distribution of the configuration on the unrevealed edges is simply
\begin{equation} \label[ineq]{ineq:just_after_tmoat}
	\pcond \bra{ \omega\vert_{E \backslash\overline{K_o}} = \boldsymbol{\cdot} \mid \mathcal F_{\tmoat} } = \p_p^{G \backslash \overline{K_o} } \leq_{\mathrm{st}} \p_p \quad \text{a.s.}
\end{equation}
When $\tfail < \abs{E}$, we showed that $\rho_{\tfail+1}$ almost surely has an endpoint belonging to $K_o$, namely $v_-$. So $\tfail < \abs{E}$ implies $\tfail < \tmoat$ almost surely. By applying this observation, \cref{ineq:just_after_tmoat}, and transitivity, we obtain
\[ \begin{split}
	\pcond \bra{ \big\lVert K_{v_+}(\omega \backslash \overline{K_o} )\big\rVert \geq \frac{\alpha}{2} \mid \mathcal F_{\tfail}} &= \econd \sqbra{ \pcond \bra{ \big\lVert K_{v_+}(\omega \backslash \overline{K_o} )\big\rVert \geq \frac{\alpha}{2} \mid \mathcal F_{\tmoat}} \mid \mathcal F_{\tfail} } \\
	&\leq \econd \sqbra{ \p_p \bra{ \den{K_{v_+}} \geq \frac{\alpha}{2} } \mid \mathcal F_{\tfail} } \\
	&= \p_p \bra{ \den{K_o} \geq \frac{\alpha}{2} } \quad \text{a.s.\! when } \tfail < \abs{E}.
\end{split} \]
When $\den{K_o} \geq \frac{\alpha}{2}$, we must have that $o \in K_1$ or $\den{K_2} \geq \frac{\alpha}{2}$. So by a union bound,
\begin{equation} \label[ineq]{ineq:second_term} \begin{split}
	\pcond \bra{ \big\lVert K_{v_+}(\omega \backslash \overline{K_o} )\big\rVert \geq \frac{\alpha}{2} \mid \mathcal F_{\tfail}} &\leq \p_p \bra{ o \in K_1 } + \p_p \bra{ \den{K_2} \geq \frac{\alpha}{2}  } \\
	&\leq \theta + h \quad \text{a.s.\! when } \tfail < \abs{E},
\end{split}\end{equation}
where $h := \p_p \bra{B}$ and $\theta := \e_p \den{K_1}$, which satisfies $ \theta = \p_p \bra{ o \in K_1 }$ by transitivity.

Plugging \cref{ineq:piv_to_split,ineq:first_term,ineq:second_term} into \cref{eq:final_manipulation} gives a lower bound on $\pcond \bra{ \rho_{\tfail+1} \text{ open} \mid \mathcal F_{\tfail} }$ when $\tfail < \abs{E}$. By contrasting this with upper bound \cref{ineq:trivial_bound_at_tfail} and expanding the definition of $q$, we deduce that
\[
	p\bra{ 1 - \theta - h - \pcond \bra{ B \mid \mathcal F_{\tfail}} } \leq p\bra{1 - \theta - \frac{2h^{1/2}}{1-\theta-h}} \quad \text{a.s.\! when } \tfail < \abs{E}.
\]
In particular, 
\[
	\pcond \bra{ B \mid \mathcal F_{\tfail}} \geq \frac{h^{1/2}}{1-\theta-h} \quad \text{a.s.\! when } \tfail < \abs{E}.
\]
By the law of total expectation, we know that $\econd \sqbra{ \pcond \bra{ B \mid \mathcal F_{\tfail}} } = \pcond \bra{ B }$. So by Markov's inequality,
\begin{equation} \label[ineq]{ineq:tfail_markov}
	\pcond \bra{ \tfail < \abs{E} } \leq \pcond \bra{ \pcond \bra{ B \mid \mathcal F_{\tfail}} \geq \frac{h^{1/2}}{1-\theta-h} } \leq \frac{1-\theta - h}{h^{1/2}} \pcond \bra{B}.
\end{equation}
By definition of $\pcond$ and $h$, we have the trivial bound
\[
	\pcond \bra{B} \leq \frac{\p_p \bra{B}}{\p_p \bra{ \den{K_o} < \alpha }} = \frac{h}{1-\p_p \bra{ \den{K_o} \geq \alpha }}.
\]
If $\den{K_o} \geq \alpha$, then $o \in K_1$ or $\den{K_2} \geq \alpha$. So similarly to the argument for \cref{ineq:second_term}, a union bound gives $\pcond \bra{ B } \leq \frac{h}{1-\theta - h}$. Plugging this into \cref{ineq:tfail_markov} yields $\pcond \bra{ \tfail < \abs{E} } \leq h^{1/2}$, as required.
\end{proof}

\section{Characterising the existence of a percolation threshold}

In this section, we prove \cref{thm:main} and \cref{cor:lower_bound}. Our first step is to establish a mean-field lower bound on the supercritical giant cluster density for sequences of graphs without molecular subsequences. This is where we use the coupling lemma from \cref{sec:coupling_lemma}.

\begin{lem} \label{lem:mean-field}
	Let $(G_n)$ be a sequence of finite connected transitive graphs with volume tending to infinity that does not contain a molecular subsequence. Fix $\eps > 0$ and let $(p_n)$ be any sequence of parameters. If $\liminf_{n \to \infty} \e_{p_n}^{G_n} \den{K_1} > 0$, then
	\[
		\liminf_{n \to \infty} \e_{(1+\eps)p_n}^{G_n} \den{K_1} \geq \frac{\eps}{1+\eps}.
	\]
\end{lem}

\begin{proof}
	Suppose for contradiction that $\liminf_{n \to \infty} \e_{p_n}^{G_n} \den{K_1} > 0$ but $\liminf_{n \to \infty} \e_{(1+\eps)p_n}^{G_n} \den{K_1} < \frac{\eps}{1+\eps}$. By passing to a suitable subsequence, we may assume that
	\begin{equation} \label[ineq]{ineq:mean_field_violated}
		\limsup_{n \to \infty} \e_{(1+\eps)p_n}^{G_n} \den{K_1} < \frac{\eps}{1+\eps}.
	\end{equation}
	Pick $\alpha >0$ such that $\e_{p_n}^{G_n} \den{K_1} \geq 2\alpha$ for all sufficiently large $n$. By Markov's inequality applied to $1 - \den{K_1}$, this implies that $\p_{p_n}^{G_n} \bra{ \den{K_1} \geq \alpha } \geq \alpha$ for all sufficiently large $n$. Now for each $n$, define
	\[
		\theta_n := \e_{(1+\eps)p_n}^{G_n} \den{K_1} \qquad \text{and} \qquad h_n := \p_{(1+\eps)p_n}^{G_n} \bra{ \den{K_1} < \alpha \text{ or } \den{K_2} \geq \frac{\alpha}{2} }.
	\]
	Since $(G_n)$ does not contain a molecular subsequence, it has the sharp-density and supercritical uniqueness properties from \cite{https://doi.org/10.48550/arxiv.2112.12778}. (See \cref{subsec:defining:molecular_seq} for more information.) So the fact that $\liminf_{n \to \infty} \p_{p_n}^{G_n} \bra{ \den{K_1} \geq \alpha } >0$ implies that
	\[
		\lim_{n \to \infty} \p_{(1+\eps)p_n}^{G_n} \bra{ \den{K_1} < \alpha } = 0 \qquad \text{and} \qquad \lim_{n \to \infty} \p_{(1+\eps)p_n}^{G_n} \bra{ \den{K_2} \geq \frac{\alpha}{2} } = 0.
	\]
	So by a union bound, $\lim_{n \to \infty} h_n = 0$. Together with \cref{ineq:mean_field_violated}, this guarantees that for all sufficiently large $n$, we have $\theta_n + h_n < 1$, allowing us to apply \cref{lem:coupling}. By passing to a suitable tail of $(G_n)$, we may assume that this holds for every $n$. So for every $n$, \cref{lem:coupling} says that there is an event $A_n \subseteq \{0,1\}^{E(G_n)}$ with
	\begin{equation} \label[ineq]{ineq:error_event_bound}
		\p_{(1+\eps)p_n}^{G_n} \bra{ A_n \mid \den{K_o} < \alpha } \leq h_n^{1/2}
	\end{equation}
	such that
	\begin{equation} \label[ineq]{ineq:main_couple}
		\p_{ \bra{1 - \theta_n - \frac{2h_n^{1/2}}{1-\theta_n-h_n}}(1+\eps)p_n }^{G_n} \leq_{\mathrm{st}} \p_{(1+\eps)p_n}^{G_n} \bra{\omega \cup \1_{A_n} = \boldsymbol{\cdot} \mid \den{K_o} < \alpha}.
	\end{equation}

	When $\p_{p_n}^{G_n} \bra{ \den{K_1} \geq \alpha } \geq \alpha$, it follows by transitivity that $\p_{p_n}^{G_n} \bra{ \den{K_o} \geq \alpha } \geq \alpha^2$. On the other hand, by \cref{ineq:error_event_bound,ineq:main_couple}, we know that
	\[
		\p_{ \bra{1 - \theta_n - \frac{2h_n^{1/2}}{1-\theta_n-h_n}}(1+\eps)p_n }^{G_n} (\den{K_o} \geq \alpha) \leq h_n^{1/2}.
	\]
	So by taking $n$ sufficiently large that both $\p_{p_n}^{G_n} \bra{ \den{K_1} \geq \alpha } \geq \alpha$ and $h_n^{1/2} < \alpha^2$, we can force
	\begin{equation} \label[ineq]{ineq:parameter_comparison}
		\bra{1 - \theta_n - \frac{2h_n^{1/2}}{1-\theta_n-h_n}}(1+\eps)p_n \leq p_n.
	\end{equation}
	However, by \cref{ineq:mean_field_violated} and the fact that $\lim_{n \to \infty} h_n = 0$,
	\[
		1 - \theta_n - \frac{2h_n^{1/2}}{1-\theta_n-h_n} > \frac{1}{1+\eps}
	\]
	for all sufficiently large $n$, contradicting \cref{ineq:parameter_comparison}.
\end{proof}

By the sharp-density property, we can convert this mean-field lower bound that holds in \emph{expectation} into one that holds \emph{with high probability}.

\begin{lem} \label{lem:mean-field_whp}
	Let $(G_n)$ be a sequence of finite connected transitive graphs with volume tending to infinity that does not contain a molecular subsequence. Fix $\eps > 0$ and let $(p_n)$ be any sequence of parameters. If $\liminf_{n \to \infty} \e_{p_n}^{G_n} \den{K_1} > 0$, then
	\[
		\p_{(1+\eps)p_n}^{G_n} \bra{ \den{K_1} \geq \frac{\eps}{1+\eps} - o(1)  } = 1 - o(1) \quad \text{as $n \to \infty$.}
	\]
\end{lem}

\begin{proof}
	Let $\delta \in (0,\eps)$ be any constant. By \cref{lem:mean-field},
	\[
		\liminf_{n \to \infty} \e_{(1+\eps-\delta)p_n}^{G_n} \den{K_1} \geq \frac{\eps-\delta}{1+\eps-\delta}.
	\]
	So by Markov's inequality applied to $1-\den{K_1}$,
	\begin{equation} \label[ineq]{ineq:mean-field_just_below}
		\liminf_{n \to \infty} \p_{(1+\eps-\delta)p_n}^{G_n} \bra{ \den{K_1} \geq \frac{\eps - \delta}{1+\eps}} \geq \frac{\delta(\eps-\delta)}{(1+\delta)(1+\eps-\delta)} > 0.
	\end{equation}
	Since $(G_n)$ does not contain a molecular subsequence, it has the sharp-density property from \cite{https://doi.org/10.48550/arxiv.2112.12778}. (Recall our discussion in \cref{subsec:defining:molecular_seq}.) So \cref{ineq:mean-field_just_below} implies that
	\[
		\lim_{n \to \infty} \p_{(1+\eps)p_n}^{G_n} \bra{ \den{K_1} \geq \frac{\eps - \delta}{1+\eps} } = 1.
	\]
	Since $\delta$ was arbitrary, the result now follows by a diagonal argument.
\end{proof}

Our next step is to extend a version of this lower bound to a simple kind of molecular sequence. The idea is to break the graphs in the sequence into their constituent \emph{atoms} then apply \cref{lem:mean-field_whp} to the sequence of atoms.

\begin{lem} \label{lem:mean-field_whp_for_molecules}
	Let $(G_n)$ be a sequence of finite connected transitive graphs with volume tending to infinity. Assume that $(G_n)$ is $m$-molecular for some $m \geq 2$ but does not contain a $k$-molecular subsequence for any $k > m$. Fix $\eps > 0$ and let $(p_n)$ be any sequence of parameters. If $\liminf_{n \to \infty} \e_{p_n}^{G_n} \den{K_1} > 0$, then
	\[
		\p_{(1+\eps)p_n}^{G_n} \bra{ \den{K_1} \geq \frac{\eps}{m(1+\eps)} - o(1)  } = 1 - o(1) \quad \text{as $n \to \infty$.}
	\]
\end{lem}

\begin{proof}
By definition of an $m$-molecular sequence, for every $n$, we can pick a set of edges $F_n \subseteq E(G_n)$ such that $F_n$ is $\operatorname{Aut} G_n$-invariant, $G_n \backslash F_n$ has $m$ connected components, and $\frac{\abs{F_n}}{\abs{V(G_n)}}$ is uniformly bounded. Let $A_n$ denote one of the connected components of $G_n \backslash F_n$, which are all necessarily isomorphic to each other and transitive. Notice that $(A_n)$ does not contain a molecular subsequence. Indeed, if $(A_n)$ contained an $r$-molecular subsequence $(A_n)_{n \in I}$ for some $r \geq 2$, then $(G_n)_{n \in I}$ would be an $r m$-molecular subsequence of $(G_n)$. (Every automorphism of $G_n$ acts on the $m$ copies of $A_n$ by permuting the copies and applying an automorphism of $A_n$ to each.)

For each finite graph $G$, let $\lambda(G)$ denote the largest eigenvalue of the adjacency matrix for $G$. Recall that when $G$ is regular, $\lambda(G) = \deg(G)$, the vertex degree of $G$. By \cref{thm:dense}, which is taken from \cite{MR2599196}, since $(G_n)$ and $(A_n)$ are sequences of dense graphs, they have percolation thresholds at $(1/\lambda(G_n))_{n \geq 1}$ and $(1/\lambda(A_n))_{n \geq 1}$ respectively. Since $\frac{\abs{F_n}}{\abs{V(G_n)}}$ is uniformly bounded, $\lim_{n \to \infty} \frac{\deg A_n}{\deg G_n} = 1$, and since every $A_n$ and $G_n$ is regular (since transitive), this means that $\lim_{n \to \infty} \frac{\lambda(A_n)}{\lambda(G_n)} = 1$. So for any $\delta >0$, since the sequence $((1+\delta)p_n)$ is supercritical for $(G_n)$, it is also supercritical for $(A_n)$. In particular,
\[
	\liminf_{n \to \infty} \e_{ (1+\delta) p_n }^{A_n} \den{K_1} > 0.
\]
Now since $(A_n)$ has no molecular subsequences, it follows by \cref{lem:mean-field_whp} that
\[
	\p_{(1+\eps)p_n}^{A_n} \bra{ \den{K_1} \geq \frac{\eps-\delta}{1+\eps} - o(1) } = 1 - o(1) \quad \text{as $n\to\infty$.}
\]
Since $\delta > 0$ was arbitrary, a diagonal argument gives
\[
	\p_{(1+\eps)p_n}^{A_n} \bra{ \den{K_1} \geq \frac{\eps}{1+\eps} - o(1) } = 1 - o(1) \quad \text{as $n\to\infty$.}
\]
The result follows because $A_n$ is a subgraph of $G_n$ and $\abs{V(G_n)} = m \abs{V(A_n)}$.
\end{proof}

We now extend this lower bound to sequences of graphs that have $m$-molecular subsequences for at most finitely many integers $m$.

\begin{lem} \label{lem:mean-field_whp_for_good_sequences}
	Let $(G_n)$ be a sequence of finite connected transitive graphs with volume tending to infinity that does not contain an $m$-molecular subsequence for any $m > M$, where $M$ is some positive integer. Fix $\eps > 0$ and let $(p_n)$ be any sequence of parameters. If $\liminf_{n \to \infty} \e_{p_n}^{G_n} \den{K_1} > 0$, then
	\[
		\p_{(1+\eps)p_n}^{G_n} \bra{ \den{K_1} \geq \frac{\eps}{M(1+\eps)} - o(1)  } = 1 - o(1) \quad \text{as $n \to \infty$.}
	\]
\end{lem}

\begin{proof}
It is enough to show that for every subsequence $(G_n)_{n \in I}$ of $(G_n)$, we can find a further subsequence $(G_n)_{n \in J}$ with $J \subseteq I$ such that
\[
		\p_{(1+\eps)p_n}^{G_n} \bra{ \den{K_1} \geq \frac{\eps}{M(1+\eps)} - o(1)  } = 1 - o(1) \quad \text{as $n \to \infty$ with $n \in J$.}
\]

Let $(G_n)_{n \in I}$ be a subsequence of $(G_n)$. If $(G_n)_{n \in I}$ does not contain a molecular subsequence, then by \cref{lem:mean-field_whp}, we know that $\p_{(1+\eps)p_n}^{G_n} \bra{ \den{K_1} \geq \frac{\eps}{1+\eps} - o(1)  } = 1 - o(1)$ as $n \to \infty$ with $n \in I$. On the other hand, if $(G_n)_{n \in I}$ does contain a molecular subsequence, then we can pick a subsequence $(G_n)_{n \in J}$ with $J \subseteq I$ that is $m$-molecular with $m \in \{2,\dots,M \}$ maximum. Then \cref{lem:mean-field_whp_for_molecules} tells us that $\p_{(1+\eps)p_n}^{G_n} \bra{ \den{K_1} \geq \frac{\eps}{m(1+\eps)} - o(1)  } = 1 - o(1)$ as $n \to \infty$ with $n \in J$.
\end{proof}

We are now ready to prove that if a sequence of graphs has $m$-molecular subsequences for at most finitely many integers $m$, then it has a percolation threshold. As outlined in \cref{subsec:proof_strategy}, the idea is to prove that the threshold for the emergence of a cluster of very slightly sublinear density is a percolation threshold.

\begin{lem} \label{lem:perco_threshold_for_good_sequences}
If a sequence of finite connected transitive graphs with volume tending to infinity contains an $m$-molecular subsequence for at most finitely many integers $m$, then it has a percolation threshold.
\end{lem}

\begin{proof}
Let $(G_n)$ be such a sequence of graphs. Let $M$ be a positive integer such that there are no $m$-molecular subsequences with $m > M$. For every density $\delta \in (0,1)$ and each index $n$, define $p_c^n(\delta)$ to be the unique parameter satisfying $\p_{p_c^n(\delta)}^{G_n} \bra{ \den{K_1} \geq \delta } = \frac{1}{2}$. Given any constants $\eps,\delta \in (0,1)$, we know by \cref{lem:mean-field_whp_for_good_sequences} that
\[
	\lim_{n \to \infty} \p_{(1+\eps) p_c^n(\delta)}^{G_n} \bra{ \den{K_1} \geq \frac{\eps}{2M} } = 1,
\]
since $\frac{\eps}{2M} < \frac{\eps}{M(1+\eps)}$. So by a diagonal argument, for every $\eps \in (0,1)$, there exists a sequence $(\delta_n^\eps)\seq$ in $(0,1)$ with $\lim_{n \to \infty} \delta_n^\eps =0$ such that
\[
	\lim_{n \to \infty} \p_{(1+\eps) p_c^n(\delta_n^\eps)}^{G_n} \bra{ \den{K_1} \geq \frac{\eps}{2M} } = 1.
\]
Now by a further diagonal argument, there exists a fixed sequence $(\delta_n)$ in $(0,1)$ with $\lim_{n \to \infty} \delta_n = 0$ such that for every $\eps \in \cubra{ \frac{1}{2}, \frac{1}{3}, \frac{1}{4},\dots }$,
\[
	\lim_{n \to \infty} \p_{(1+\eps) p_c^n(\delta_n)}^{G_n} \bra{ \den{K_1} \geq \frac{\eps}{2M} } = 1.
\]
We claim that the sequence of parameters $(p_n)$ given by $p_n := p_c^n(\delta_n)$ has the required properties.

To verify the subcritical condition, suppose for contradiction that ${\lim_{n \to \infty} \p_{(1-\eps)p_n}^{G_n} \bra{ \den{K_1} \geq \alpha } \not=0}$ for some $\eps,\alpha \in (0,1)$. By passing to a suitable subsequence, we may assume that
\[
	\liminf_{n \to \infty} \p_{(1-\eps)p_n}^{G_n} \bra{ \den{K_1} \geq \alpha } > 0.
\]
Then by \cref{lem:mean-field_whp_for_good_sequences},
\[
	\lim_{n \to \infty} \p_{p_n}^{G_n} \bra{ \den{K_1} \geq \frac{\eps}{2M} }=1.
\]
This contradicts the fact that for every $n$ that is sufficiently large to ensure $\delta_n \leq \frac{\eps}{2M}$,
\[
	\p_{p_n}^{G_n} \bra{ \den{K_1} \geq \frac{\eps}{2M} } \leq \p_{p_n}^{G_n} \bra{ \den{K_1} \geq \delta_n } = \frac{1}{2}.
\]

To verify the supercritical condition, fix $\eps > 0$. Pick any $\eps' \in \cubra{ \frac{1}{2}, \frac{1}{3}, \frac{1}{4},\dots }$ with $\eps' < \eps$. Then by construction of $(\delta_n)$,
\[
	\p_{(1+\eps)p_n}^{G_n} \bra{ \den{K_1} \geq \frac{\eps'}{2M} } \geq \p_{(1+\eps')p_n}^{G_n} \bra{ \den{K_1} \geq \frac{\eps'}{2M} } = 1 - o(1) \quad \text{as $n \to \infty$.} \qedhere
\]
\end{proof}

We now verify that these sequences of graphs - those that have $m$-molecular subsequences for at most finitely many integers $m$ - are the \emph{only} sequences to have a percolation threshold. This follows from the following corollary of \cref{thm:sup} from \cite{https://doi.org/10.48550/arxiv.2112.12778} that characterises the \emph{supercritical existence property}.

\begin{defn}
	We say that $(G_n)$ has the \emph{supercritical existence property} if for every supercritical sequence of parameters $(q_n)$,
	\[
		\lim_{n \to \infty} \p_{q_n}^{G_n} \bra{ \den{K_1} \geq \alpha } = 1 \qquad \text{for some $\alpha > 0$}.
	\]
\end{defn}

\begin{cor}[Easo and Hutchcroft, 2021] \label{cor:sep}
A sequence of finite connected transitive graphs with volume tending to infinity has the supercritical existence property if and only if it contains an $m$-molecular
subsequence for at most finitely many integers $m$.
\end{cor}

Notice that if a sequence of finite connected transitive graphs with volume tending to infinity has a percolation threshold, then it automatically has the supercritical existence property. So the `only if' direction of \cref{cor:sep} immediately implies the following lemma.

\begin{lem} \label{lem:supermolecular_no_perco_threshold}
	If a sequence of finite connected transitive graphs with volume tending to infinity contains an $m$-molecular subsequence for infinitely many integers $m$, then it does not have a percolation threshold.
\end{lem}

Our main result, \cref{thm:main}, now follows by combining \cref{lem:perco_threshold_for_good_sequences,lem:supermolecular_no_perco_threshold}. We conclude this section by deducing \cref{cor:lower_bound} from \cref{lem:mean-field_whp_for_good_sequences}.

\begin{proof}[Proof of \cref{cor:lower_bound}]
For every $\delta \in (0,\eps)$, we know by definition of a percolation threshold that $\liminf_{n \to \infty} \e_{(1+\delta)p_n}^{G_n} \den{K_1} > 0$. So by \cref{lem:mean-field_whp_for_good_sequences},
\[
	\p_{(1+\eps)p_n}^{G_n} \bra{ \den{K_1} \geq \frac{\eps - \delta}{M(1+\eps)} - o(1) } = 1 - o(1) \quad \text{as $n \to \infty$.}
\]
Since $\delta$ was arbitrary, the result follows by a diagonal argument. 
\end{proof}

\section{Bounding the threshold location}
In this section, we give a proof of \cref{prop:sharp_bounds}. This is simply the observation that existing arguments immediately imply bounds on the location of a percolation threshold (when it exists), which happen to be best-possible. The lower bound is a completely standard path-counting argument that we only include for completeness. The upper bound comes from bounding the percolation threshold by the connectivity threshold. As explained in \cite{MR4262440}, we can estimate the connectivity threshold thanks to an upper bound by Karger and Stein on the number of \emph{approximate} minimum cutsets in a graph \cite{MR1409212}.

The lower bound is sharp in the classical case of the complete graphs \cite{MR125031}. In fact, it is sharp for a very wide range of examples, including arbitrary dense regular graphs \cite{MR2599196}, the hypercubes \cite{MR671140} (which are sparse yet have unbounded vertex degrees), and expanders with high girth and bounded vertex degrees \cite{MR2073175}. However, it is less obvious that the upper bound is optimal in the sense that the constant '2' cannot be improved. For this, see the \emph{fat cycles} construction in \cite{MR4262440}.

\begin{proof}[Proof of \cref{prop:sharp_bounds}]
We start with the lower bound. This argument is just an optimisation of the proof of Lemma 2.8 in \cite{https://doi.org/10.48550/arxiv.2112.12778}. Fix $\eps \in (0,1)$. It suffices to check that $\bra{\frac{1-\eps}{d_n-1}}\seq$ is not supercritical along any subsequence. There are at most $d_n (d_n-1)^{r-1}$ simple paths of length $r$ starting at a particular vertex $o$ in each $G_n$. So for every $n$,
\[
	\e_{ \frac{1-\eps}{d_n-1} }^{G_n} \abs{K_o} \leq \sum_{r \geq 0} d_n(d_n-1)^{r-1} \bra{\frac{1-\eps}{d_n-1}}^{r} \leq \frac{2}{\eps}.
\]
In particular, $\lim_{n \to \infty} \e_{ \frac{1-\eps}{d_n-1} }^{G_n} \den{K_o} = 0$, and hence by transitivity, $\lim_{n \to \infty} \e_{ \frac{1-\eps}{d_n-1} } \den{K_1} = 0$.

We now prove the upper bound. Every finite connected transitive graph has edge connectivity equal to its vertex degree \cite{MR291004}. So by Theorem 4.1 from \cite{MR4262440},
\[
	\lim_{n \to \infty } \p_{(2+o(1))\frac{\log \abs{V(G_n)}}{d_n} }^{G_n} \bra{\omega \text{ is connected}} = 1.
\]
In particular, $\bra{\frac{2\log \abs{V(G_n)}}{d_n}}\seq$ is not subcritical along any subsequence.
\end{proof}

\section{Acknowledgements}
I thank Tom Hutchcroft for helpful comments on earlier drafts and for encouraging me to pursue this question of mine. I also thank an anonymous referee for their useful feedback and suggestions, which improved the clarity of the paper.

\printbibliography

\end{document}